\documentclass[12pt]{article}
\usepackage{amsmath}
\usepackage{graphicx}
\usepackage{enumerate}
\usepackage{natbib}
\usepackage{url} 
\usepackage[utf8]{inputenc}
\usepackage{epsfig}
\usepackage{latexsym}
\usepackage{amsthm, amssymb, mathrsfs, amsfonts}
\usepackage{varioref}
\usepackage[table]{xcolor}
\usepackage{hhline}
\usepackage{mathtools}
\usepackage[inline]{enumitem}
\usepackage{longtable}
\usepackage[inline]{enumitem}
\usepackage{bbm}
\usepackage{lineno}
\usepackage{multirow}

\newcommand{\blind}{1}

\allowdisplaybreaks

\renewcommand{\theequation}{\arabic{equation}}
\newcommand{\beq}{\begin{equation}}
	\newcommand{\eeq}{\end{equation}}

\newtheorem{lemma}{Lemma}

\newtheorem{prop}{Proposition}

\def\ba{\begin{array}}
	\def\ea{\end{array}}

\newcommand{\Prob}{\operatorname{P}}

\newcommand{\MI}{M\hspace{-1.0mm}I}

\newcommand\numberthis{\addtocounter{equation}{1}\tag{\theequation}}

\begin{document}

\def\spacingset#1{\renewcommand{\baselinestretch}%
{#1}\small\normalsize} \spacingset{1}


\if1\blind
{
  \title{\bf A Normal Test for Independence via Generalized Mutual Information}
  \author{Jialin Zhang\thanks{Department of Mathematics and Statistics, Mississippi State University}\hspace{.2cm}\\
    and \\
    Zhiyi Zhang\thanks{Department of Mathematics and Statistics, University of North Carolina at Charlotte}}
  \maketitle
} \fi

\if0\blind
{
  \bigskip
  \bigskip
  \bigskip
  \begin{center}
    {\LARGE\bf A Normal Test for Independence via Generalized Mutual Information}
\end{center}
  \medskip
} \fi

\bigskip
\begin{abstract}
Testing hypothesis of independence between two random elements on a joint alphabet is a fundamental exercise in statistics. Pearson's chi-squared test is an effective test for such a situation when the contingency table is relatively small. General statistical tools are lacking when the contingency data tables are  large or sparse. A test based on generalized mutual information is derived and proposed in this article. The new test has two desired theoretical properties. First, the test statistic is asymptotically normal under the hypothesis of independence; consequently it does not require the knowledge of the row and column sizes of the contingency table. Second, the test is consistent and therefore it would detect any form of dependence structure in the general alternative space given a sufficiently large sample. In addition, simulation studies show that the proposed test converges faster than Pearson's chi-squared test when the contingency table is large or sparse.
\end{abstract}

\noindent%
\vfill

\newpage
\spacingset{1.9} 

\section{Introduction and Summary} 
Let $(X,Y)$ be a pair of random elements on a joint alphabet, 
$\mathscr{X}\times \mathscr{Y}=\{(l_{i},m_{j});1\leq i\leq I, 1\leq j\leq J\}$, with a joint probability distribution,  $\mathbf{p}_{\scalebox{0.7}{X,Y}}=\{p_{i,j};   1\leq i\leq I, 1\leq j\leq J    \}$,   
and the two marginal distributions, $\mathbf{p}_{\scalebox{0.7}{X}}=\{p_{i,\cdot}=\sum_{j\geq 1}p_{i,j};i\geq 1\}$ and $\mathbf{p}_{\scalebox{0.7}{Y}}=\{p_{\cdot,j}=\sum_{i\geq 1}p_{i,j};j\geq 1\}$, for $X$ and $Y$ respectively. 
Consider one the most fundamental problems of statistics: testing the hypothesis of independence between $X$ and $Y$, denoted $H_{0}: X\perp Y$ versus $H_{a}: X\not\perp Y$. Let an identically and independently distributed (iid) sample of size $n$ be represented by the empirical distribution, that is, $\hat{\mathbf{p}}_{\scalebox{0.7}{X,Y}}=\{\hat{p}_{i,j}=f_{i,j}/n\}$ where $f_{i,j}$ is the observed frequency of letter $(l_{i},m_{j})$. Let 
$\hat{\mathbf{p}}_{\scalebox{0.7}{X}}=  \{\hat{p}_{i,\cdot}=\sum_{j\geq 1}f_{i,j}/n;i\geq 1\}$ and 
$\hat{\mathbf{p}}_{\scalebox{0.7}{Y}}= \{\hat{p}_{\cdot,j}=\sum_{i\geq 1}f_{i,j}/n;j\geq 1\}$ be the two observed marginal relative frequencies of $X$ and $Y$. 

In statistical practice, the standard procedure for such a setting is the well-studied Pearson's chi-squared test, which is based on the fact that, as $n\rightarrow \infty$,
\beq \label{pearson}
T=n\sum_{i,j}\hat{p}_{i,\cdot}\hat{p}_{\cdot,j}\left(\frac{\hat{p}_{i,j}-\hat{p}_{i,\cdot}\hat{p}_{\cdot,j}   }{\hat{p}_{i,\cdot}\hat{p}_{\cdot,j}}\right)^2 \stackrel{p}{\longrightarrow} \chi^2_{(I-1)(J-1)}
\eeq   where $ \chi^2_{(I-1)(J-1)}$ is a chi-squared random variable with degrees of freedom $\nu=(I-1)(J-1)$. The Pearson's chi-squared test is an effective tool for relatively small 2-way contingency tables. However it is not without discomforting issues in practice, particularly when it is applied to a large or sparse contingency table. 

One of the issues in practice is when $I$ and $J$ are unknown. In such a case, it is difficult to fix the reference distribution in (\ref{pearson}). A popular adjustment in practice is to replace $I$ and $J$ with observed numbers of rows and columns, $\hat{I}$ and $\hat{J}$. However, such an adjustment lacks theoretical support and the reference distribution used may be quite far away from the asymptotic chi-squared distribution. Another commonly encountered issue in a large contingency table is the occurrence of low-frequency cells. 
Given the fact that the essence of the asymptotic behavior of Pearson's chi-squared statistic is  the asymptotic normality of $\sqrt{n} \left[  (\hat{p}_{i,j}-\hat{p}_{i,\cdot} \hat{p}_{\cdot,j}) / (\hat{p}_{i,\cdot}\hat{p}_{\cdot,j}) \right]  $
in each cell, many low or zero frequency cells in a large contingency table could negatively impact the performance of the test, mostly in the form of a much inflated Type I error probability.  This is a long-standing issue considered by many in the existing literature. 

A popular adjustment to offset the low or zero frequency cells, when applying a Pearson-type chi-squared test, is to combine cells to increase the cell frequency. A well-known rule of thumb, often thought to be suggested by R.A. Fisher, is to combine cells into new cells such that the combined frequencies are at least five. However, in applying Pearson's chi-squared test for independence, it is not clear how this adjustment may be done. To assure independence under $H_{0}$, the adjustment must be made by combining low-frequency rows and low-frequency columns, respectively, due to the invariance of independence under any row permutation and/or column permutation. However, by doing so, it is not guaranteed that all new cells would see enough frequencies. Only by combining the rows and columns further, can the new cells then see meaningfully higher frequencies. When this is the case, the re-aggregation becomes somewhat arbitrary. Even if this could be done, the following two points of concern remain.  
\begin{enumerate}
	\item Aggressive re-aggregation of cells could greatly reduce the number of (observed) degrees of freedom and consequently could shift the reference distribution to one that is far away from $\chi^2_{(I-1)(J-1)}$. 
	\item Aggressive re-aggregation of cells could cause local dependence between $X$ and $Y$, manifested in fine structures of the joint distribution, to be inadvertently buried, and hence could deprive Pearson's chi-squared test a chance to detect such a dependence. 
\end{enumerate} 

Consider a simple but amplified illustrative example of concept as follows. Let 
\begin{align*} 
	\mathbf{p}_{\scalebox{0.7}{X}}&=\{p_{i,\cdot}; i=1,\cdots,11\}=\{0.9, 0.01,\cdots,0.01\}, \\
	\mathbf{p}_{\scalebox{0.7}{Y}}&=\{p_{\cdot,j}; j=1,\cdots,11\}=\{0.9, 0.01,\cdots,0.01\}. 
\end{align*} Under independence, the joint distribution of $(X,Y)$ is given in $\mathbf{p}_{0}$ below.
\[
\mathbf{p}_{0}
=\left(\begin{array}{c|ccc}
	0.8100 & 0.0090 & \cdots & 0.0090 \\ \hline
	0.0090 & 0.0001 & \cdots & 0.0001 \\
	0.0090 & 0.0001 & \cdots & 0.0001 \\
	\vdots & \vdots & \vdots\hspace{2pt}\vdots\hspace{2pt}\vdots & \vdots \\
	0.0090 & 0.0001 & \cdots & 0.0001
\end{array}\right)
=\left(\begin{array}{cc}
	P_{1,1} & P_{1,2} \\
	P_{2,1} & P_{2,2}
\end{array}\right)
\] 
Summing up all probabilities in $P_{2,2}$, $\sum_{i=2}^{11}\sum_{j=2}^{11}p_{i,j}=0.1$. Let the total mass of $0.1$ be 
redistributed uniformly only on the diagonal of $P_{2,2}$, augmenting $P_{2,2}$ into 
\[
P_{2,2}^*=\left(\begin{array}{cccc}
	0.001 &  &  &   \\
	& 0.001 &  &   \\
	&  & \ddots &   \\
	&  &  & 0.001  
\end{array}\right)
\] with all off-diagonal elements being zeros. Let
\[
\mathbf{p}_{a}=\left(\begin{array}{cc}
	P_{1,1} & P_{1,2} \\
	P_{2,1} & P_{2,2}^*
\end{array}\right)
\] and let it be assumed that $(X,Y)$ follows the joint distribution $\mathbf{p}_{a}$. Clearly, $X$ and $Y$ are not independent. 
Suppose there is a sample with a sufficiently large $n$, such that, all cells with a positive probability see $f_{i,j}\geq 5$. That however does not change the fact that the observed frequencies for cells corresponding to the zero-probability locations in $P_{2,2}^*$ are zeros. In applying the usual re-aggregation of the contingency table, by means of combining rows and columns, one would not be able to end up with all cell frequencies greater than or equal to five, unless all rows of $i=2,\cdots, 11$ are lumped together and all columns of $j=2,\cdots,11$ are lumped together. However by doing this, the underlying joint distribution becomes 
\[     
\mathbf{p}_{a}^*=\left(\begin{array}{cc}
	0.81 & 0.09 \\
	0.09 & 0.01
\end{array}\right)
\] under which $X$ and $Y$ are independent. In this example, it is evident that the aggregated data would imply $\hat{I}=\hat{J}=2$, far away from the $I=J=11$. It is also evident that the data aggregation would completely erase the fine dependence structure in $P_{2,2}^*$ and would leave no chance for the dependence to be detected.  

As the need to extend Pearson's chi-squared test to accommodate data in a large or a sparse contingency table increases, many studies have been reported and less stringent rules of thumb have been proposed. The main guideline for the chi-squared test focuses on the (estimated) expected cell counts in the contingency table, following from \cite{cochran1952chi2}, \cite{cochran1954some}, \cite{agresti2003categorical}, and \cite{yates1999practice}. The widely accepted general rule of thumb (referred to below as the Rule) is:
(a) at least 80$\%$ the expected counts are five or greater, and (b) all individual expected counts are one or greater.  This however is still very stringent. In the example given above, Part (b) alone requires on average $n\times 0.0001\geq 1$, or $n\approx 10000$, or much greater if $f_{i,j}\geq 1$ is required. In fact, in the example above, $f_{i,j}\geq 1$ cannot be satisfied for each and every $(i,j)$. 

To alleviate the above-mentioned difficulties, a new test of independence in a contingency table is proposed in this article. The proposed test has at least two desirable properties.   First, the asymptotic distribution of the test statistic is normal under the independence assumption, and consequently neither the test statistic nor its asymptotic distribution requires the knowledge of $I$ and $J$. Second, the test is consistent and therefore it would detect any form of dependence structure in the general alternative space given a sufficiently large sample. In addition, empirical evidence shows that the proposed test converges faster than Pearson's chi-squared test when the contingency table is large or sparse.

There are five sections in this article. The main results leading to the proposed test are discussed in Section 2. In Section 3, several simulation studies are presented.  A few concluding remarks are given in Section 4. The article ends with Appendix where a few proofs are found.

\section{Toward a Normal Test for Independence}
Consider Shannon's entropy, introduced in \cite{shannon1948mathematical}, for an random element $Z$ assuming a label in a countable  alphabet
$\mathscr{Z}=\{\ell_{k};k\geq 1\}$ with probability distribution $\mathbf{p}=\{p_{k};k\geq 1\}$, 
$H(Z)=-\sum_{k\geq 1}p_{k}\ln p_{k}$, 
and Shannon's mutual information of $X$ and $Y$, $\MI(X,Y)=H(X)+H(Y)-H(X,Y)$. One of the most important utilities of Shannon's mutual information is based on the fact that $\MI(X,Y)=0$ if and only if $X$ and $Y$ are independent. 
The plug-in estimator of $\MI(X,Y)$, $\widehat{\MI}=\hat{H}(X)+\hat{H}(Y)-\hat{H}(X,Y)$,
where $\hat{H}(X)=-\sum_{i\geq 1,}\hat{p}_{i,\cdot}\ln \hat{p}_{i,\cdot}$, 
$\hat{H}(Y)=-\sum_{j\geq 1}\hat{p}_{\cdot,j}\ln \hat{p}_{\cdot,j}$, and $\hat{H}(X,Y)=-\sum_{i\geq 1,j\geq 1}\hat{p}_{i,j}\ln \hat{p}_{i,j}$, is well-studied, and it is known, under mild conditions, that $\sqrt{n}(\widehat{\MI}-\MI)\stackrel{p}{\rightarrow} N(0,\sigma^2_{m})$ where $\sigma^2_{m}$ may be estimated by a consistent estimator.  Many details of the said fact may be found in \cite{zhang2016statistical}. However the mild conditions include $\MI>0$. When $\MI=0$, $\sqrt{n}(\widehat{\MI}-\MI)= \sqrt{n}\widehat{\MI} $ degenerates, but on the other hand, $2n  \widehat{\MI} \stackrel{p}{\rightarrow}\chi^2_{(I-1)(J-1)}$, and the derivation of this fact may be found in \cite{wilks1938large}.

Toward proposing the normal test, let the notion of escort distributions be introduced. In the context of thermodynamics, \cite{beck1995thermodynamics} defines an escort distribution as an induced distribution based on an original distribution, $\mathbf{p}=\{p_{k};k\geq 1\}$, by means of a positive function $g(p)>0$ on $(0,1)$. Let $p_{k}^*=g(p_{k})/\sum_{i\geq 1}g(p_{i})$ for each $k\geq 1$. $\mathbf{p}^*=\{p^*_{k};k\geq 1\}$ is referred to as an escort distribution. The notion of escort distributions is increasingly adopted in recent years as a means of describing random behaviors of different components in a complex system, each of which scans an underlying distribution $\mathbf{p}=\{p_{k};k\geq 1\}$ via a possibly different function $g(p)$. For a specific function form, $g(p)=p^{\lambda}$ where $\lambda>0$ is a parameter, the resulting escort distribution, 
$\mathbf{p}^*=\{p_{k}^*=p_{k}^\lambda/\sum_{i\geq 1}p^{\lambda}_{i};k\geq 1\}$,
is known as a power escort distribution.

Applying the power escort transformation to the joint distribution $\mathbf{p}_{\scalebox{0.7}{X,Y}}$, the resulting distribution is
\beq \label{escortPower}
\mathbf{p}^*_{\scalebox{0.7}{X,Y}}=\left\{p_{i,j}^*=\frac{p_{i,j}^\lambda}{\sum_{s\geq 1,t\geq 1}p^{\lambda}_{s,t}};i\geq 1, j\geq 1 \right\}.
\eeq 
Let $X^*$ and $Y^*$ be a pair of random elements on the same joint alphabet $\mathscr{X}\times\mathscr{Y}$ according to the joint distribution
$\mathbf{p}^*_{\scalebox{0.7}{X,Y}}$ of (\ref{escortPower}). The following lemma is due to \cite{zhang2020generalized}.
\begin{lemma}\label{lemma1}
	Given $\lambda>0$, 
	\begin{enumerate}
		\item $\mathbf{p}_{\scalebox{0.7}{X,Y}}$ and $\mathbf{p}^*_{\scalebox{0.7}{X,Y}}$ uniquely determine each other, that is,  $\mathbf{p}_{\scalebox{0.7}{X,Y}} \Leftrightarrow  \mathbf{p}^*_{\scalebox{0.7}{X,Y}}$; and 
		\item $X$ and $Y$ are independent if and only if $X^*$ and $Y^*$ are independent, that is, 
		$X\perp Y \Leftrightarrow X^* \perp Y^*$.
	\end{enumerate}
\end{lemma}
By Part 2 of Lemma \ref{lemma1}, the null hypothesis, $H_{0}: X\perp Y$ may then be stated equivalently as $H_{0}:X^*\perp Y^*$, that is, $\MI(X^*,Y^*)=0$, or letting $c_{\lambda, \mathbf{p}}=\sum_{s,t} p_{s,t}^\lambda$, 
\beq \label{IndepDef1}
\sum_{i,j}   \left(\frac{p^{\lambda}_{i,j}}{ c_{\lambda, \mathbf{p}}}\right)     
\ln\left( \frac{     p^{\lambda}_{i,j}} {c_{\lambda, \mathbf{p}}} \right)   
- 
\sum_{i} \left[  \left( \frac{  \sum_{t}   p^{\lambda}_{i,t}} {c_{\lambda, \mathbf{p}}} \right) 
\ln \left( \frac{  \sum_{t}   p^{\lambda}_{i,t}} {c_{\lambda, \mathbf{p}}} \right) \right]
- 
\sum_{j} \left[   \left( \frac{\sum_{s}   p^{\lambda}_{s,j}}{ c_{\lambda, \mathbf{p}}  } \right)
\ln \left( \frac{  \sum_{s}   p^{\lambda}_{s,j}} { c_{\lambda, \mathbf{p}} } \right)  \right]=0.
\eeq

On the other hand, let it be observed that under $H_{0}$, 
\begin{align*}
	c_{\lambda, \mathbf{p}}&=\sum_{s,t} p_{s,t}^\lambda = \left( \sum_{s} p_{s,\cdot}^\lambda\right)\left( \sum_{t} p_{\cdot,t}^\lambda \right),  \label{insert01} \numberthis \\
	\sum_{i} \left[  \left( \frac{  \sum_{t}   p^{\lambda}_{i,t}} {c_{\lambda, \mathbf{p}}} \right) 
	\ln \left( \frac{  \sum_{t}   p^{\lambda}_{i,t}} {c_{\lambda, \mathbf{p}}} \right) \right]
	& =\sum_{i} \left[  \left( \frac{ p^{\lambda}_{i,\cdot}} { \sum_{s} p_{s,\cdot}^\lambda} \right) 
	\ln  \left( \frac{ p^{\lambda}_{i,\cdot}} { \sum_{s} p_{s,\cdot}^\lambda} \right)    \right],  \label{insert02} \numberthis  \\
	\sum_{j} \left[   \left( \frac{\sum_{s}   p^{\lambda}_{s,j}}{ c_{\lambda, \mathbf{p}}  } \right)
	\ln \left( \frac{  \sum_{s}   p^{\lambda}_{s,j}} { c_{\lambda, \mathbf{p}} } \right)  \right]
	& =\sum_{j} \left[  \left( \frac{ p^{\lambda}_{\cdot,j}} { \sum_{t} p_{\cdot,t}^\lambda} \right) 
	\ln  \left( \frac{ p^{\lambda}_{\cdot,j}} { \sum_{t} p_{\cdot,t}^\lambda} \right)    \right].   \label{insert03} \numberthis \\
\end{align*}
Let it also be noted that 
\begin{enumerate}
	\item (\ref{IndepDef1}) is a necessary and sufficient condition, that is, the equality of (\ref{IndepDef1})  holds if and only if $X\perp Y$,
	\item the equalities in  (\ref{insert01}), (\ref{insert02}) and (\ref{insert03}) do not necessarily hold in general but under the assumption of $H_{0}$. 
\end{enumerate} Adding and subtracting the left-hand sides of (\ref{insert02}) and (\ref{insert03}) to and from (\ref{IndepDef1}), another restatement of $H_{0}$ is obtained below.
\begin{align*} \label{IndepDef2}
	& \left\{\sum_{i,j}   \left(\frac{p^{\lambda}_{i,j}}{ c_{\lambda, \mathbf{p}}}\right)     
	\ln\left( \frac{     p^{\lambda}_{i,j}} {c_{\lambda, \mathbf{p}}} \right)    
	-    
	\sum_{i} \left[  \left( \frac{ p^{\lambda}_{i,\cdot}} { \sum_{s} p_{s,\cdot}^\lambda} \right) 
	\ln  \left( \frac{ p^{\lambda}_{i,\cdot}} { \sum_{s} p_{s,\cdot}^\lambda} \right)    \right]
	-
	\sum_{j} \left[  \left( \frac{ p^{\lambda}_{\cdot,j}} { \sum_{t} p_{\cdot,t}^\lambda} \right) 
	\ln  \left( \frac{ p^{\lambda}_{\cdot,j}} { \sum_{t} p_{\cdot,t}^\lambda} \right)    \right]\right\}
	\\
	&
	\\
	& + \left\{    
	\sum_{i} \left[  \left( \frac{ p^{\lambda}_{i,\cdot}} { \sum_{s} p_{s,\cdot}^\lambda} \right) 
	\ln  \left( \frac{ p^{\lambda}_{i,\cdot}} { \sum_{s} p_{s,\cdot}^\lambda} \right)  \right]
	+
	\sum_{j} \left[  \left( \frac{ p^{\lambda}_{\cdot,j}} { \sum_{t} p_{\cdot,t}^\lambda} \right) 
	\ln  \left( \frac{ p^{\lambda}_{\cdot,j}} { \sum_{t} p_{\cdot,t}^\lambda} \right)    \right] \right. \\
	& \hspace{2em}
	\left.
	- 
	\sum_{i} \left[  \left( \frac{  \sum_{t}   p^{\lambda}_{i,t}} {c_{\lambda, \mathbf{p}}} \right) 
	\ln \left( \frac{  \sum_{t}   p^{\lambda}_{i,t}} {c_{\lambda, \mathbf{p}}} \right) \right]     
	- 
	\sum_{j} \left[   \left( \frac{\sum_{s}   p^{\lambda}_{s,j}}{ c_{\lambda, \mathbf{p}}  } \right)
	\ln \left( \frac{  \sum_{s}   p^{\lambda}_{s,j}} { c_{\lambda, \mathbf{p}} } \right)  \right] 
	\right\}  \\
	& =0. \numberthis
\end{align*} Writing the terms within the curly brackets in (\ref{IndepDef2}) as $T_A$ and $T_{B}$, (\ref{IndepDef2}) becomes
\beq \label{IndepDef3}
T_{A} + T_{B}=0. 
\eeq A natural test for independence would be to statistically check the value of the left-hand-side in (\ref{IndepDef1}), or that in  (\ref{IndepDef2}), or that in (\ref{IndepDef3}), and assess the statistical evidence against that value being zero. Consider the plug-in estimator of the the left-hand-side of (\ref{IndepDef3}), by 
replacing $p_{i,j}$ with $\hat{p}_{i,j}=f_{i,j}/n$ for every pair $(i,j)$, $p_{i,\cdot}$ with $\hat{p}_{i,\cdot}=\sum_{j}f_{i,j}/n$ for every $i$, and $p_{\cdot,j}$ with $\hat{p}_{\cdot,j}=\sum_{i}f_{i,j}/n$ for every $j$, resulting in plug-in estimators of $T_{A}$ and $T_{B}$, denoted by $\hat{T}_{A}$ and $\hat{T}_{B}$.

By \cite{wilks1938large}, $2n(\hat{T}_{A}+\hat{T}_{B})\stackrel{p}{\rightarrow} \chi^2_{(I-1)(J-1)}$ under $H_{0}$, which implies that  
\beq \label{core01}\sqrt{n}( \hat{T}_{A}+\hat{T}_{B}  ) = \frac{1}{2\sqrt{n}}[2n  ( \hat{T}_{A}+\hat{T}_{B}  ) ]  \stackrel{p}{\longrightarrow}0.
\eeq
The following proposition is the keystone of the test to be proposed.

\begin{prop}\label{prop1} Suppose neither of the two underlying marginal distributions, $\mathbf{p}_{\scalebox{0.7}{X}}$ and 
	$\mathbf{p}_{\scalebox{0.7}{Y}}$, is uniform, and $H_{0}: X\perp Y$ holds. 
	$\sqrt{n}\hat{T}_{A}\stackrel{d}{\longrightarrow} N(0,\sigma^2)$ as $n\rightarrow \infty$, where $\sigma^2>0$ is a positive constant depending on the parameter $\lambda\in (0,1)\cup (1,\infty)$.
\end{prop} A proof of Proposition \ref{prop1} is given in Appendix. 

Let $N=\lim_{n\rightarrow\infty}\sqrt{n}\hat{T}_{A}$ denote the normal random variable under the conditions of Proposition \ref{prop1}. By (\ref{core01}) and Lemma \ref{prop1}, 
$-N=\lim_{n\rightarrow\infty}\sqrt{n}\hat{T}_{B}$, where $N$ is the same random variable as in $N=\lim_{n\rightarrow\infty}\sqrt{n}\hat{T}_{A}$. 

\begin{prop}\label{prop2} Under the conditions of Proposition \ref{prop1}, 
	\begin{enumerate}
		\item $Z_{A}=\sqrt{n}\hat{T}_{A}/\hat{\sigma} \stackrel{d}{\longrightarrow}N(0,1)$,
		\item $Z_{B}=\sqrt{n}\hat{T}_{B}/\hat{\sigma} \stackrel{d}{\longrightarrow}N(0,1)$, and
		\item $Z_{AB}=1_{[|Z_{A}|\geq |Z_{B}|]}Z_{A}+1_{[|Z_{A}|< |Z_{B}|]}Z_{B}  \stackrel{d}{\longrightarrow}N(0,1)$, 
	\end{enumerate} where $\hat{\sigma}^2$ is the variance given in (\ref{varianceTA}) with all $p_{i,j}$ replaced by $\hat{p}_{i,j}$ for $i\geq 1$ and $j\geq 1$.
\end{prop}

At least three tests for $H_{0}: X\perp Y$  are feasible according to Proposition \ref{prop2}.
\begin{enumerate}
	\item Test 1: $H_{0}$ is rejected if $Z_{A}<-z_{\alpha/2}$ or $Z_{A}>z_{\alpha/2}$;
	\item Test 2: $H_{0}$ is rejected if $Z_{B}<-z_{\alpha/2}$ or $Z_{B}>z_{\alpha/2}$; and
	\item Test 3: $H_{0}$ is rejected if $Z_{AB}<-z_{\alpha/2}$ or $Z_{AB}>z_{\alpha/2}$
\end{enumerate} where $\alpha\in (0,1)$ is a prefixed constant, $z_{\alpha/2}$ is the $100\times (1-\alpha/2)$ th percentile of the standard normal distribution. The test based on $Z_{AB}$ is the proposed test, and it is a consistent test as described in Proposition \ref{prop3} below.

\begin{prop}\label{prop3} Suppose neither of the two underlying marginal distributions, $\mathbf{p}_{\scalebox{0.7}{X}}$ and 
	$\mathbf{p}_{\scalebox{0.7}{Y}}$, is uniform.  Then  
	\[
	\lim_{n\rightarrow \infty}\Prob(Z_{AB}\in (-\infty,-z_{\alpha/2})\cup (z_{\alpha/2},\infty)|H_{a}) =1.
	\]
\end{prop} A proof of Proposition \ref{prop3} is given in Appendix.

\section{Simulations}
The performance of the proposed test is assessed by simulations. Numerous simulation studies are carried out for cases with various forms of underlying distributions. In each case, the proposed test is compared against Pearson's chi-squared test, with degrees of freedom $(I-1)(J-1)$ and $(\hat{I}-1)(\hat{J}-1)$ respectively, with  
six levels of sample size, $n=30, 100, 500, 1000, 1500$, and $2000$. The results summarized in Table \ref{table} are representative of the general trends observed and therefore are presented below.

The sequence of five pairs of $H_{0}$ and $H_{a}$ is specifically constructed as follows. The example of the $11\times 11$ 
contingency  table in Section 1 is one of such cases. 
In that example, the contingency table has $I=11=1+10$ rows and $J=11=1+10$ columns. 
A more general distribution may be described as follows. For a given value $p\in (0,1)$, let both of the row and the column marginal  be  $\{1-p, p/(I-1), p/(I-1),\cdots, p/(I-1) \}$.
A pair of $H_{0}$ and $H_{a}$ may be constructed as follows. 

Under $H_{0}$, a joint distribution is constructed as follows.
\beq \label{SparseDistributionH0}
\mathbf{p}_{0}= \left(\begin{array}{cccc}
	(1-p)^2 & p(1-p)/(I-1) & \cdots & p(1-p)/(I-1) \\
	p(1-p)/(I-1) & p^2/(I-1)^2 &  \cdots & p^2/(I-1)^2  \\
	p(1-p)/(I-1) & p^2/(I-1)^2 & \cdots & p^2/(I-1)^2 \\
	\vdots & \vdots & \vdots\hspace{2pt}\vdots\hspace{2pt}\vdots  &\vdots \\
	p(1-p)/(I-1) & p^2/(I-1)^2 & \cdots & p^2/(I-1)^2 \\
\end{array}\right).
\eeq
The joint distribution of (\ref{SparseDistributionH0}) is reconstructed, first by summing all entries in the lower-right $(I-1)(J-1)$ sub-matrix and then redistributing the sum on the diagonal of the sub-matrix uniformly, resulting in

\[    \mathbf{p}_{a}= \left(\begin{array}{cccc}
	(1-p)^2 & p(1-p)/(I-1) & \cdots & p(1-p)/(I-1) \\
	p(1-p)/(I-1) & p^2/(I-1) &  \cdots & 0    \\
	p(1-p)/(I-1) &  0 &   \ddots  &   0 \\
	\vdots &   \vdots &   \vdots\hspace{2pt}\vdots\hspace{2pt}\vdots  &  0 \\
	p(1-p)/(I-1) & 0  &  \cdots & p^2/(I-1) \\
\end{array}\right).
\] Thus, $H_{0}: \mathbf{p}_{0}$ versus $H_{a}:  \mathbf{p}_{a}$ becomes a pair. Letting the parameter $p$ take on values, 0.5, 0.6, 0.7, 0.8, and 0.9, respectively, the dependence structure in $H_{a}$ becomes weaker and weaker. 

The results of the simulation studies are reported in Table \ref{table}, each based on one hundred thousand replicates of iid sample of indicated size $n$, and $\lambda=2$. 

Referring to Table \ref{table},  the distribution with $1-p=0.5$ on the top corresponds to the strongest contrast between $H_{0}$ and $H_{a}$ among all five cases in the table. For $n=30$ and $n=100$, none of the three test statistics converges satisfactorily under $H_{0}$ at $\alpha=0.01$. For $n\geq 500$, both the proposed test and the Pearson's chi-squared test with observed degrees of freedom converge satisfactorily under $H_{0}$, and provide very good power at $H_{a}$. However, the Pearson's chi-squared test with theoretical degrees of freedom does not converge under $H_{0}$. 

The distribution with $1-p=0.6$ provides a less strong contrast between $H_{0}$ and $H_{a}$. For $n\geq 500$, both the proposed test and the Pearson's chi-squared test with observed degrees of freedom converge satisfactorily under $H_{0}$, and provide very good power at $H_{a}$. However, the Pearson's chi-squared test with theoretical degrees of freedom has a very inflated Type I error probability. One may also notice that the Pearson's chi-squared test with observed degrees of freedom perhaps converges a little slower than the proposed test under $H_{0}$.

The distribution with $1-p=0.7$ provides an even less strong contrast between $H_{0}$ and $H_{a}$. For $n\geq 500$, only the proposed test converges satisfactorily under $H_{0}$, and provide meaningful power at $H_{a}$. However, both of the Pearson's chi-squared tests have inflated Type I error probability. 

The distribution with $1-p=0.8$ provides a weak contrast between $H_{0}$ and $H_{a}$. For $n\geq 500$, the proposed test converges satisfactorily under $H_{0}$, and provide little power at $H_{a}$. On the other hand, both of the Pearson's chi-squared tests have very inflated Type I error probability. One may also notice that the Pearson's chi-squared test with observed degrees of freedom perhaps converges a little slower than the proposed test under $H_{0}$.

The distribution with $1-p=0.9$ provides a very weak contrast between $H_{0}$ and $H_{a}$. So much so that even for $n\geq 500$, the proposed test converges very slowly under $H_{0}$, and provides essentially no power at $H_{a}$. On the other hand, both of the Pearson's chi-squared tests have a total breakdown under $H_{0}$. 

The following three major points are observed in Table \ref{table}, as well as in other simulation studies investigated but not presented herewithin.
\begin{enumerate}
	\item In small or dense contingency tables, both Pearson's chi-squared tests converge quickly under $H_{0}$ and are generally more powerful than the proposed test.
	\item In larger and sparse contingency tables,  both Pearson's chi-squared tests converge much slower under $H_{0}$ and tend to have a much higher probability of Type I error than what is intended. 
\end{enumerate}
All things considered, in practice, if the Rule is satisfied, then Pearson's chi-squared test is recommended, otherwise the proposed test of this article is recommended, provided that $n IJ\geq 5$ or more practically $n \hat{I} \hat{J}\geq 5$. In that regard, the proposed test is not meant to replace Pearson's test in all circumstances, but only when Pearson's test  is judged not appropriate by the current rule-of-thumb criteria. 

\section{Remarks}
The idea of the article may be explained simply by the convergence rate of $\widehat{\MI}(X^*,Y^*)$ under $H_{0}$.
$\widehat{\MI}(X^*,Y^*)$ is $n$-convergent, that is, 
\begin{align*}
	n \widehat{\MI}(X^*,Y^*) &=n[(H(X^*)+H(Y^*))+(-H(X^*,Y^*))]\rightarrow \chi^2_{(I-1)(J-1)}/2, \mbox{ and} \\
	\sqrt{n}\widehat{\MI}(X^*,Y^*)&=\sqrt{n}[(H(X^*)+H(Y^*))+(-H(X^*,Y^*))]\rightarrow 0.
\end{align*} In other words, under $H_{0}$, the difference between the two additive parts of 
$\widehat{\MI}(X^*,Y^*) $ approaches zero very fast. 
However by inserting a zero in the form of four terms, as in (\ref{IndepDef2}),  $\widehat{\MI}(X^*,Y^*) $ is wedged into two terms, $\hat{T}_{A}$ and $\hat{T}_{B}$, both of which approach zero under $H_{0}$, but are at a much slower rate, that is, $\sqrt{n}$-convergent. This insertion is almost literally a keystone, splitting $\sqrt{n} \widehat{\MI}(X^*,Y^*)$ into two random variables and eking out asymptotic normality of $Z_{A}$, of $Z_{B}$, and hence of $Z_{AB}$ under $H_{0}$.
The immediate advantages of the normality of $Z_{AB}$ are that the knowledge of $I$ and $J$ is not required when testing $H_{0}$, and that the test statistic seems to converge faster than Pearson's chi-squared test under $H_{0}$. On the other hand, the statistical assessment of $H_{0}$ is done by two separate random pieces instead of one, as in $n\widehat{\MI}$, and some efficiency may be lost as evidenced by the simulation studies. The observed loss of power may be considered a cost for more generality, that is,  the knowledge of $I$ and $J$ is not required.

It is to be noted that the concept of escort distributions is essential in the arguments leading to the proposed test. Only when $\lambda\neq 1$, the inserts, the left-hand-sides of (\ref{insert02}) and (\ref{insert03}), would enable an positive variance in (\ref{varianceTA}), and hence the asymptotic normality of $Z_{AB}$. 

It is also to be highlighted that the proposed test based on $Z_{AB}$ is a consistent test as stated in Proposition \ref{prop3}. This fact lends the utility of the proposed test in the general alternative space. That is to say, provided a sufficiently large sample, any form of dependent structure between $X$ and $Y$ will be detected. The test is proposed herewithin in the form of a two-sided test due to its generality. For specific forms of dependence structures between $X$ and $Y$, some of the tests based on $Z_{A}$, $Z_{B}$ or $Z_{AB}$, one-sided or two-sided, may have better performance in terms of faster convergence under $H_{0}$ and higher power under $H_{a}$ than others.  This provides a potentially fruitful direction for further investigation.

\section{Appendix}
\begin{proof}[Proof of Proposition \ref{prop1}]
	It may be verified that for each pair $(i,j)$, $i = 1, \cdots, I-1$, and $j = 1, \cdots, J-1$,
	\begin{align*}
		\frac{\partial T_A}{\partial p_{i, j}} = & \quad (\ln p^*_{I,.}+1)\frac{\lambda p^*_{I,.}}{p_{I,.}} - (\ln p^*_{i,.} +1 )\frac{\lambda p^*_{i,.}}{p_{i,.}} + \lambda \left[\frac{p^*_{I,.}}{p_{I,.}} p^{\lambda -1}_{.,J} - \frac{p^*_{i,.}}{p_{i,.}}\right](H(X^*) - 1) \\
		& + (\ln p^*_{.,J}+1)\frac{\lambda p^*_{.,J}}{p_{.,J}} - (\ln p^*_{.,j} +1 )\frac{\lambda p^*_{.,j}}{p_{.,j}} + \lambda \left[\frac{p^*_{.,J}}{p_{.,J}} p^{\lambda -1}_{I,.} - \frac{p^*_{.,j}}{p_{.,j}}\right](H(Y^*) - 1) \\
		& + (\ln p^*_{i,j}) \frac{\lambda p^*_{i,j}}{p_{i,j}} + \lambda \left(\frac{p^*_{i,j}}{p_{i,j}} - \frac{p^*_{I, J}}{p_{I, J}}\right)H(X^*, Y^*) - \lambda (\ln p^*_{I,J}) \frac{p^*_{I,J}}{p_{I,J}} 
	\end{align*} where $p^*_{i,j}$ is defined in (\ref{escortPower}), $p^*_{\cdot,j}=\sum_{i}p^*_{i,j}$, and $p^*_{\cdot,j}=\sum_{j}p^*_{i,j}$; 
	
	for $i = I$, and $j = 1, \dots, J-1$,
	\begin{align*}
		\frac{\partial T_A}{\partial p_{i=I, j}} = & \quad  \lambda \frac{p^*_{I,.}}{p_{I,.}} \left( p^{\lambda -1}_{.,J} - p^{\lambda -1}_{.,j}\right)(H(X^*) - 1) \\
		& + (\ln p^*_{.,J}+1)\frac{\lambda p^*_{.,J}}{p_{.,J}} - (\ln p^*_{.,j} +1 )\frac{\lambda p^*_{.,j}}{p_{.,j}} 
		+ \lambda \left( \frac{p^*_{.,J}}{p_{.,J}} p^{\lambda -1}_{I,.} - \frac{p^*_{.,j}}{p_{.,j}} \right) (H(Y^*) - 1) \\
		& + (\ln p^*_{I,j}) \frac{\lambda p^*_{I,j}}{p_{I,j}} + \lambda \left(\frac{p^*_{I,j}}{p_{I,j}} - \frac{p^*_{I, J}}{p_{I, J}}\right)H(X^*, Y^*) - \lambda (\ln p^*_{I,J}) \frac{p^*_{I,J}}{p_{I,J}}
	\end{align*}
	
	and for $i = 1, \dots, I-1$, and $j = J$,
	\begin{align*}
		\frac{\partial T_A}{\partial p_{i, J}} = & \quad (\ln p^*_{I,.}+1)\frac{\lambda p^*_{I,.}}{p_{I,.}} - (\ln p^*_{i,.} +1 )\frac{\lambda p^*_{i,.}}{p_{i,.}} + \lambda \left( \frac{p^*_{I,.}}{p_{I,.}} p^{\lambda -1}_{.,J} - \frac{p^*_{i,.}}{p_{i,.}} \right)(H(X^*) - 1) \\
		& + \lambda \frac{p^*_{.,J}}{p_{.,J}} \left( p^{\lambda -1}_{I,.} - p^{\lambda -1}_{i,.}\right)(H(Y^*) - 1) \\
		& +( \ln p^*_{i,J}) \frac{\lambda p^*_{i,J}}{p_{i,J}} + \lambda \left(\frac{p^*_{i,J}}{p_{i,J}} - \frac{p^*_{I, J}}{p_{I, J}}\right)H(X^*, Y^*) - \lambda (\ln p^*_{I,J})  \frac{p^*_{I,J}}{p_{I,J}} 
	\end{align*}
	
	Let
	\begin{align*}
		&\mathbf{v}=(v_{1},\cdots,v_{IJ-1})^{\tau}=\left(p_{1, 1}, p_{1, 2}. \cdots, p_{1, J}, p_{2, 1}, p_{2,2},\cdots, p_{2, J}, \cdots, p_{I,1}, p_{I,2}, \cdots, p_{I, J-1}\right)^{\tau}, \\
		&\hat{\mathbf{v}} =(\hat{v}_{1},\cdots,\hat{v}_{IJ-1})^{\tau}   =\left(\hat{p}_{1, 1}, \hat{p}_{1, 2}. \cdots, \hat{p}_{1, J}, \hat{p}_{2, 1}, \hat{p}_{2,2},\cdots, \hat{p}_{2, J}, \cdots, \hat{p}_{I,1}, \hat{p}_{I,2}, \cdots, \hat{p}_{I, J-1}\right)^{\tau},
	\end{align*}
	noting specifically the implied enumeration of the indexes of $\mathbf{v}$ corresponds to the arrangement of $(i,j)$ as in $p_{i,j}$, that is,
	\beq \label{indexes}
	\{1,2,\cdots,IJ-1\}=\{(1,1),\cdots, (1,J), (2,1), \cdots, (2,J), \cdots, (I-1,J), (I,1),\cdots,(I,J-1)\}.
	\eeq
	Let the gradient of $T_{A}$ with respect to $p_{i,j}$ for all $(i,j)$ be denoted by 
	$
	\nabla= \{ \partial T_{A} / \partial p_{i,j} \}
	$ with the index arrangement given in (\ref{indexes}). 
	
	It follows that $\sqrt{n}(\hat{\mathbf{v}}-\mathbf{v}) \stackrel{d}{\rightarrow} M V N(\mathbf{0}, \Sigma(\mathbf{v}))$, where $\Sigma(\mathbf{v})$ is the $(IJ-1) \times(IJ-1)$ covariance matrix given by
	\begin{align*}
		\Sigma=\Sigma(\mathbf{v})=\left(\begin{array}{cccc}
			v_{1}\left(1-v_{1}\right) & -v_{1} v_{2} & \cdots & -v_{1} v_{IJ-1} \\
			-v_{2} v_{1} & v_{2}\left(1-v_{2}\right) & \cdots & -v_{2} v_{IJ-1} \\
			\vdots & \vdots & \vdots\hspace{2pt}\vdots\hspace{2pt}\vdots & \vdots \\
			-v_{IJ-1} v_{1} & -v_{IJ-1} v_{2} & \cdots & v_{IJ-1}\left(1-v_{IJ-1}\right)
		\end{array}\right) 
	\end{align*}
	According to the first-order delta method, $\sqrt{n}\hat{T}_{A}\stackrel{d}{\longrightarrow} N(0,\sigma^2)$, as $n\rightarrow \infty$, where
	\beq \label{varianceTA}
	\sigma^2 = \nabla^{\tau} \Sigma \nabla.
	\eeq
\end{proof}

\begin{proof}[Proof of Proposition \ref{prop3}]
	Under $H_{a}: X\not\perp Y$, by Lemma \ref{lemma1}, $T_{A}+T_{B}=\mu_{\mathbf{p}}>0$ where $T_{A}$ and $T_{B}$ are as in (\ref{IndepDef3}). By the respective consistencies of plug-in estimators of $T_{A}$, $T_{B}$ and $\sigma^2$, 
	$\hat{T}_{A}+\hat{T}_{B}\stackrel{p}{\rightarrow}\mu_{\mathbf{p}}>0$, and hence
	$\sqrt{n}\hat{T}_{A}/\hat{\sigma}+\sqrt{n}\hat{T}_{A}/\hat{\sigma} =Z_{A}+Z_{B}\stackrel{p}{\rightarrow}\infty$, which in turn implies that at least one of $Z_{A}$ and $Z_{B}$ is carried above all bounds in probability, which finally implies that $Z_{AB}\stackrel{p}{\rightarrow}\infty$, as $n\rightarrow \infty$.
\end{proof}

\begin{table}[ht]
	\begin{center}
	\resizebox{.75\textwidth}{!}{%
	\begin{tabular}{|c|l|cc|cc|cc|}
		\hline
		Distribution                & Sample Size & \multicolumn{2}{c|}{$Z_{AB}: H_0, H_a$} & \multicolumn{2}{c|}{$\chi^2_{(\hat{I}-1)(\hat{J}-1)}: H_0, H_a$} & \multicolumn{2}{c|}{$\chi^2_{(I-1)(J-1)}: H_0, H_a$} \\ \hline
		\multirow{6}{*}{$1-p=0.50$} & $n$ = 30    & \multicolumn{1}{c|}{0.2002}  & 0.3235 & \multicolumn{1}{c|}{0.0632}              & 0.4377               & \multicolumn{1}{c|}{0.0831}        & 0.0694           \\ \cline{2-8} 
		& $n$ = 100   & \multicolumn{1}{c|}{0.0561}  & 0.7739 & \multicolumn{1}{c|}{0.0373}              & 0.9951              & \multicolumn{1}{c|}{0.0439}        & 0.9940        \\ \cline{2-8} 
		& $n$ = 500   & \multicolumn{1}{c|}{0.0145}  & 1.0000       & \multicolumn{1}{c|}{0.0124}              & 1.0000                    & \multicolumn{1}{c|}{0.2058}        & 1.0000          \\ \cline{2-8} 
		& $n$ = 1000  & \multicolumn{1}{c|}{0.0122}  & 1.0000       & \multicolumn{1}{c|}{0.0113}              & 1.0000                    & \multicolumn{1}{c|}{0.1437}        & 1.0000              \\ \cline{2-8} 
		& $n$ = 1500  & \multicolumn{1}{c|}{0.0122}  & 1.0000       & \multicolumn{1}{c|}{0.0109}              & 1.0000                    & \multicolumn{1}{c|}{0.10841}        & 1.0000              \\ \cline{2-8} 
		& $n$ = 2000  & \multicolumn{1}{c|}{0.0113}  & 1.0000       & \multicolumn{1}{c|}{0.0098}              & 1.0000                    & \multicolumn{1}{c|}{0.0858}        & 1.0000              \\ \hline
		\multirow{6}{*}{$1-p=0.60$} & $n$ = 30    & \multicolumn{1}{c|}{0.0540}     & 0.0862 & \multicolumn{1}{c|}{0.0968}              & 0.2811              & \multicolumn{1}{c|}{0.0021}        & 0.0133              \\ \cline{2-8} 
		& $n$ = 100   & \multicolumn{1}{c|}{0.0071}  & 0.0965 & \multicolumn{1}{c|}{0.0764}              & 0.8760              & \multicolumn{1}{c|}{0.0997}        & 0.8515       \\ \cline{2-8} 
		& $n$ = 500   & \multicolumn{1}{c|}{0.0087}   & 0.9267 & \multicolumn{1}{c|}{0.0183}              & 1.0000                    & \multicolumn{1}{c|}{0.0767}        & 1.0000              \\ \cline{2-8} 
		& $n$ = 1000  & \multicolumn{1}{c|}{0.0105}  & 0.9997  & \multicolumn{1}{c|}{0.0143}              & 1.0000                    & \multicolumn{1}{c|}{0.0412}        & 1.0000              \\ \cline{2-8} 
		& $n$ = 1500  & \multicolumn{1}{c|}{0.0103}  & 1.0000       & \multicolumn{1}{c|}{0.0123}              & 1.0000                    & \multicolumn{1}{c|}{0.0316}        & 1.0000              \\ \cline{2-8} 
		& $n$ = 2000  & \multicolumn{1}{c|}{0.0102}  & 1.0000       & \multicolumn{1}{c|}{0.0120}              & 1.0000                   & \multicolumn{1}{c|}{0.0263}        & 1.0000              \\ \hline
		\multirow{6}{*}{$1-p=0.70$} & $n$ = 30    & \multicolumn{1}{c|}{0.0040}  & 0.0054 & \multicolumn{1}{c|}{0.1352}              & 0.2118              & \multicolumn{1}{c|}{0.0002}         & 0.0020          \\ \cline{2-8} 
		& $n$ = 100   & \multicolumn{1}{c|}{0.0003}   & 0.0011 & \multicolumn{1}{c|}{0.1400}              & 0.5563              & \multicolumn{1}{c|}{0.0966}        & 0.4720        \\ \cline{2-8} 
		& $n$ = 500   & \multicolumn{1}{c|}{0.0081}  & 0.1324  & \multicolumn{1}{c|}{0.0320}              & 1.0000                    & \multicolumn{1}{c|}{0.0320}        & 1.0000              \\ \cline{2-8} 
		& $n$ = 1000  & \multicolumn{1}{c|}{0.0096}  & 0.4032  & \multicolumn{1}{c|}{0.0200}              & 1.0000                    & \multicolumn{1}{c|}{0.0200}        & 1.0000              \\ \cline{2-8} 
		& $n$ = 1500  & \multicolumn{1}{c|}{0.0102}   & 0.6541 & \multicolumn{1}{c|}{0.0168}              & 1.0000                    & \multicolumn{1}{c|}{0.0168}        & 1.0000              \\ \cline{2-8} 
		& $n$ = 2000  & \multicolumn{1}{c|}{0.0104}  & 0.8196 & \multicolumn{1}{c|}{0.0151}              & 1.0000                    & \multicolumn{1}{c|}{0.0151}        & 1.0000             \\ \hline
		\multirow{6}{*}{$1-p=0.80$} & $n$ = 30    & \multicolumn{1}{c|}{0.0021}  & 0.0027 & \multicolumn{1}{c|}{0.1697}              & 0.1941              & \multicolumn{1}{c|}{0.00213}        & 0.0027        \\ \cline{2-8} 
		& $n$ = 100   & \multicolumn{1}{c|}{0.0000}        & 0.0000       & \multicolumn{1}{c|}{0.2082}              & 0.3369              & \multicolumn{1}{c|}{0.0997}        & 0.1977        \\ \cline{2-8} 
		& $n$ = 500   & \multicolumn{1}{c|}{0.0086}  & 0.0189 & \multicolumn{1}{c|}{0.0768}              & 0.9426              & \multicolumn{1}{c|}{0.0767}        & 0.9426        \\ \cline{2-8} 
		& $n$ = 1000  & \multicolumn{1}{c|}{0.0104}  & 0.0389 & \multicolumn{1}{c|}{0.0412}              & 0.9999              & \multicolumn{1}{c|}{0.0412}        & 0.9999        \\ \cline{2-8} 
		& $n$ = 1500  & \multicolumn{1}{c|}{0.0108}  & 0.0586  & \multicolumn{1}{c|}{0.0316}              & 1.0000                    & \multicolumn{1}{c|}{0.0316}        & 1.0000              \\ \cline{2-8} 
		& $n$ = 2000  & \multicolumn{1}{c|}{0.0109}  & 0.0813 & \multicolumn{1}{c|}{0.0263}              & 1.0000                    & \multicolumn{1}{c|}{0.0263}        & 1.0000              \\ \hline
		\multirow{6}{*}{$1-p=0.90$} & $n$ = 30    & \multicolumn{1}{c|}{0.0831}  & 0.0822 & \multicolumn{1}{c|}{0.2326}              & 0.2337              & \multicolumn{1}{c|}{0.0832}        & 0.0822        \\ \cline{2-8} 
		& $n$ = 100   & \multicolumn{1}{c|}{0.0000}  & 0.0000 & \multicolumn{1}{c|}{0.2386}              & 0.2519              & \multicolumn{1}{c|}{0.0439}        & 0.0532        \\ \cline{2-8} 
		& $n$ = 500   & \multicolumn{1}{c|}{0.0213}  & 0.0222 & \multicolumn{1}{c|}{0.2120}               & 0.4037               & \multicolumn{1}{c|}{0.2058}        & 0.3966        \\ \cline{2-8} 
		& $n$ = 1000  & \multicolumn{1}{c|}{0.0202}  & 0.0227 & \multicolumn{1}{c|}{0.1437}              & 0.6408              & \multicolumn{1}{c|}{0.1437}        & 0.6407        \\ \cline{2-8} 
		& $n$ = 1500  & \multicolumn{1}{c|}{0.0162}  & 0.0212 & \multicolumn{1}{c|}{0.1084}              & 0.8330              & \multicolumn{1}{c|}{0.1084}        & 0.8330        \\ \cline{2-8} 
		& $n$ = 2000  & \multicolumn{1}{c|}{0.0158}  & 0.0206 & \multicolumn{1}{c|}{0.0854}              & 0.9385              & \multicolumn{1}{c|}{0.0854}        & 0.9385        \\ \hline
	\end{tabular}%
}
	\end{center}

	\caption{Simulated Convergence and Power Comparison, $I=J=11$, $\alpha=0.01$ }
	\label{table}
\end{table}

\bibliographystyle{agsm}
\bibliography{Bibliography-MM-MC}
\end{document}